\documentclass[11pt, twoside]{article}
\usepackage{latexsym}
\usepackage{cite}
\usepackage{amsmath}\def\rad{\mbox{rad}}
\usepackage{amssymb}\def\End{\mbox{End}}
\usepackage[all]{xy}
\usepackage{amsfonts}
\usepackage{verbatim}
\usepackage{amsthm}
\usepackage{mathrsfs}
\usepackage{epsfig}
\usepackage{xy}
\usepackage{array}
\usepackage{stmaryrd}
\usepackage{graphicx,color}
\usepackage{xcolor}
\usepackage[colorlinks=true,linkcolor=blue,citecolor=blue]{hyperref}
\usepackage{tikz}
\usetikzlibrary{arrows,calc}
\usepackage{etex}
\usepackage{mathdots}
\usepackage{float}
\usepackage{graphics}
\usepackage{pdflscape}
\usepackage{extarrows}
\usepackage{anysize,hyperref}
\input xypic
\xyoption{all}
\usepackage{bm}

\usepackage[perpage,symbol]{footmisc}
\usepackage{setspace}
\def\D{\mathscr{D}}
\def\P{\mathcal{P}}
\def\I{\mathcal{I}}

\def\C{\mathscr{C}}
\def\E{\mathbb{E}}

\def\s{\mathfrak{s}}
\def\id{\mathrm{id}}
\def\op{^\mathrm{op}}
\def\Ab{\mathsf{Ab}}
\def\del{\delta}
\def\dr{\ar@{->}[r]}

\def\Im{\mbox{Im}}\def\Ker{\mbox{Ker}}
\def\add{\mbox{add}}
\def\Ext{\mbox{Ext}}\def\Sub{\mbox{Sub}}\def\sub{\mbox{sub}}
\def\Hom{\mbox{Hom}}

\newcommand{\CC}{{\bf{C}}^{n+2}_{\C}}
\newcommand{\Id}{\operatorname{Id}}
\newcommand{\ov}{\overset}
\newcommand{\lra}{\longrightarrow}
\newcommand{\co}{\colon}
\newcommand{\uas}{^{\ast}}            
\newcommand{\sas}{_{\ast}}
\newcommand{\Xd}{\langle X_{\bullet},\del\rangle}  

\newcommand{\ush}{^\sharp}           
\newcommand{\ssh}{_\sharp}
\SelectTips{cm}{10}

\usepackage{setspace}
\begin{document}
\baselineskip=15pt
\title{\Large{\bf Auslander-Reiten-Serre duality for $\bm{n}$-exangulated categories\\[1mm] \footnotetext{\hspace{-1em} Jian He was supported by the National Natural Science Foundation of China (Grant No. 12171230). Panyue Zhou was supported by the National Natural Science Foundation of China (Grant No. 11901190) and the Scientific Research Fund of Hunan Provincial Education Department (Grant No. 19B239).} }}
\medskip
\author{Jian He, Jing He and Panyue Zhou}

\date{}

\maketitle
\def\blue{\color{blue}}
\def\red{\color{red}}

\newtheorem{theorem}{Theorem}[section]
\newtheorem{lemma}[theorem]{Lemma}
\newtheorem{corollary}[theorem]{Corollary}
\newtheorem{proposition}[theorem]{Proposition}
\newtheorem{conjecture}{Conjecture}
\theoremstyle{definition}
\newtheorem{definition}[theorem]{Definition}
\newtheorem{question}[theorem]{Question}
\newtheorem{remark}[theorem]{Remark}
\newtheorem{remark*}[]{Remark}
\newtheorem{example}[theorem]{Example}
\newtheorem{example*}[]{Example}
\newtheorem{condition}[theorem]{Condition}
\newtheorem{condition*}[]{Condition}
\newtheorem{construction}[theorem]{Construction}
\newtheorem{construction*}[]{Construction}

\newtheorem{assumption}[theorem]{Assumption}
\newtheorem{assumption*}[]{Assumption}

\baselineskip=17pt
\parindent=0.5cm

\begin{abstract}
\begin{spacing}{1.2}
 Let $(\C,\E,\s)$ be an $\Ext$-finite, Krull-Schmidt and $k$-linear $n$-exangulated category with $k$ a
commutative artinian ring. In this note, we prove that $\C$ has Auslander-Reiten-Serre duality if and only if $\C$ has Auslander-Reiten $n$-exangles. Moreover, we also give an equivalent condition for the existence of Serre duality (which is a special type of Auslander-Reiten-Serre duality). Finally, assume further that $\C$ has Auslander-Reiten-Serre duality. We exploit a bijection triangle, which involves the restricted Auslander bijection and the Auslander-Reiten-Serre duality.  \\[0.2cm]
 \end{spacing}
\end{abstract}

\pagestyle{myheadings}
\markboth{\rightline {\scriptsize J. He, J. He and P. Zhou }}
         {\leftline{\scriptsize  Auslander-Reiten-Serre duality for $n$-exangulated categories}}

\section{Introduction}
Recently, Nakaoka--Palu \cite{NP} introduced an extriangulated category as a simultaneous
generalization of exact categories and triangulated categories. An extriangulated category
consists of a triple $(\C,\E,\s)$, where $\C$ is an additive category,
$$\mathbb{E}\colon \C^{\rm op}\times \C \rightarrow \Ab~~\mbox{($\Ab$ is the category of abelian groups)}$$
is an additive bifunctor and $\s$ is so called a realization of $\E$, which designates the class of conflations. Afterwards, Herschend--Liu--Nakaoka \cite{HLN} defined $n$-exangulated categories as a higher dimensional analogue of extriangulated categories.
It gives a common generalization of $n$-exact categories ($n$-abelian categories are also $n$-exact categories) in the sense of
Jasso \cite{Ja} and $(n+2)$-angulated categories in the sense of Geiss--Keller--Oppermann \cite{GKO}. However, there are some other examples of $n$-exangulated categories which are neither $n$-exact nor $(n+2)$-angulated, see \cite{HLN, LZ,HZZ2}.

Auslander-Reiten theory was introduced by Auslander and Reiten in \cite{AR1,AR2}. Since its introduction, Auslander-Reiten theory has become a fundamental tool for
studying the representation theory of artin algebras with a great impact  in other areas such as algebraic geometry and algebraic topology.
Later it has been generalized to these situation of exact categories \cite{Ji}, triangulated categories \cite{H,RV} and its subcategories \cite{AS,J}
and some certain additive categories \cite{L,J,S} by many authors. Iyama, Nakaoka and Palu \cite{INP} developed  Auslander--Reiten theory for extriangulated categories.
This unifies Auslander--Reiten theories in exact categories and triangulated categories independently.

Let $\C$ be an $\Ext$-finite, Krull-Schmidt and $k$-linear additive category with $k$ a
commutative artinian ring. Reiten-Van den Bergh, Iyama-Nakaoka-Palu and Zhou prove that $\C$ has an ``Auslander-Reiten-Serre duality'' if and only if $\C$ has ``Auslander-Reiten sequences" for triangulated, extriangulated, $(n+2)$-angulated categories, respectively, see \cite{RV,{INP},Z}. Our first main result unify and extend their results.

\begin{theorem}{\rm (see Theorem \ref{theorem1} for details)}
Let $\C$ be an $\Ext$-finite, Krull-Schmidt and $k$-linear $n$-exangulated category. Then the following statements are equivalent.
\begin{enumerate}
\item[$(1)$] $\C$ has Auslander-Reiten $n$-exangles.
\item[$(2)$] $\C$ has an Auslander-Reiten-Serre duality.

\end{enumerate}
 \end{theorem}

Serre duality is a special type of Auslander-Reiten-Serre duality. Krause, Chen-Le and Zhao-Tan-Huang  give an equivalent condition for the existence of Serre duality for triangulated, abelian, extriangulated categories, respectively, see \cite{Ch,{K},ZTH}. Our second main result unify and extend their results.

\begin{theorem}{\rm (see Theorem \ref{theorem2} for details)}
Let $\C$ be an $\Ext$-finite, Krull-Schmidt and $k$-linear $n$-exangulated category. Then $\C$ has Serre duality  if and only if $\C$ has right determined deflations and left determined inflations.
 \end{theorem}
The Auslander bijection was originally used to reduce the study of morphisms to submodules, see \cite{ARS,R}. In particular, it was shown that the Auslander bijection holds true in dualizing varieties over a commutative artin ring and in the category of finitely generated modules over an artin algebra respectively, see \cite{Ch,K}. Inspired by the comparison between the Auslander bijections and the Auslander-Reiten theory, Chen \cite{C} exploited a bijection triangle, which involves the Auslander
bijections, universal extensions and the Auslander-Reiten duality in an abelian
category having Auslander-Reiten duality. Recently, Zhao--Tan--Huang \cite{ZTH} extended Chen's result to the $\Ext$-finite, Krull-Schmidt $k$-linear extriangulated category $\C$. Namely, they gave a commutative bijection triangle, which showed that the restricted Auslander bijection holds true under the assumption that $\C$ has Auslander-Reiten-Serre duality. Our third main result show that Zhao--Tan--Huang's result has a higher counterpart:
\begin{theorem}\rm{ (see Theorem \ref{thrm} for details)} Suppose that $(\C,\E,\s)$ is an $\Ext$-finite, Krull-Schmidt and $k$-linear $n$-exangulated category having Auslander-Reiten-Serre duality. For any $X, Y\in\C $,  the following
triangle is commutative
\rm$$ \xymatrix{
&  \sub_{{\End}_{\C}(\tau_{n}^{-}X)^\mathrm{op}}\underline{\C}(\tau_{n}^{-}X,Y) & \\
{{_X}[}\rightarrow Y\rangle_{\rm def}={^{\tau_{n}^{-}X}[}\rightarrow Y\rangle_{\rm def}\ar[ur]^{\eta_{\tau_{n}^{-}X,Y}}  \ar[rr]^{\xi_{X,Y}} &  &\sub_{{\End}_{\C}(X)}{\E}(Y,X)\ar[ul]_{\Upsilon_{X,Y}}.
}
$$
In particular, we have the restricted Auslander bijection at $Y$ relative to $\tau_{n}^{-}X$
$$\eta_{\tau_{n}^{-}X,Y}: {^{\tau_{n}^{-}X}[}\rightarrow Y\rangle_{\rm def}\rightarrow  \sub_{{\End}_{\C}(\tau_{n}^{-}X)^\mathrm{op}}\underline{\C}(\tau_{n}^{-}X,Y) ,              $$
which is an isomorphism of posets.

\end{theorem}

This article is organized as follows. In Section 2, we review some elementary definitions and facts on $n$-exangulated categories. In Section 3, we prove our first and second main results. In Section 4, we prove our third main result.

\section{Preliminaries}
In this section, we briefly review basic concepts and results concerning $n$-exangulated categories.

{ For any pair of objects $A,C\in\C$, an element $\del\in\E(C,A)$ is called an {\it $\E$-extension} or simply an {\it extension}. We also write such $\del$ as ${}_A\del_C$ when we indicate $A$ and $C$. The zero element ${}_A0_C=0\in\E(C,A)$ is called the {\it split $\E$-extension}. For any pair of $\E$-extensions ${}_A\del_C$ and ${}_{A'}\del{'}_{C'}$, let $\delta\oplus \delta'\in\mathbb{E}(C\oplus C', A\oplus A')$ be the
element corresponding to $(\delta,0,0,{\delta}{'})$ through the natural isomorphism $\mathbb{E}(C\oplus C', A\oplus A')\simeq\mathbb{E}(C, A)\oplus\mathbb{E}(C, A')
\oplus\mathbb{E}(C', A)\oplus\mathbb{E}(C', A')$.

For any $a\in\C(A,A')$ and $c\in\C(C',C)$,  $\E(C,a)(\del)\in\E(C,A')\ \ \text{and}\ \ \E(c,A)(\del)\in\E(C',A)$ are simply denoted by $a_{\ast}\del$ and $c^{\ast}\del$, respectively.

Let ${}_A\del_C$ and ${}_{A'}\del{'}_{C'}$ be any pair of $\E$-extensions. A {\it morphism} $(a,c)\colon\del\to{\delta}{'}$ of extensions is a pair of morphisms $a\in\C(A,A')$ and $c\in\C(C,C')$ in $\C$, satisfying the equality
$a_{\ast}\del=c^{\ast}{\delta}{'}$.}

\begin{definition}\cite[Definition 2.7]{HLN}
Let $\bf{C}_{\C}$ be the category of complexes in $\C$. As its full subcategory, define $\CC$ to be the category of complexes in $\C$ whose components are zero in the degrees outside of $\{0,1,\ldots,n+1\}$. Namely, an object in $\CC$ is a complex $X_{\bullet}=\{X_i,d^X_i\}$ of the form
\[ X_0\xrightarrow{d^X_0}X_1\xrightarrow{d^X_1}\cdots\xrightarrow{d^X_{n-1}}X_n\xrightarrow{d^X_n}X_{n+1}. \]
We write a morphism $f_{\bullet}\co X_{\bullet}\to Y_{\bullet}$ simply $f_{\bullet}=(f_0,f_1,\ldots,f_{n+1})$, only indicating the terms of degrees $0,\ldots,n+1$.
\end{definition}

\begin{definition}\cite[Definition 2.11]{HLN}\label{def93}
By Yoneda lemma, any extension $\del\in\E(C,A)$ induces natural transformations
\[ \del\ssh\colon\C(-,C)\Rightarrow\E(-,A)\ \ \text{and}\ \ \del\ush\colon\C(A,-)\Rightarrow\E(C,-). \]
For any $X\in\C$, these $(\del\ssh)_X$ and $\del\ush_X$ are given as follows.
\begin{enumerate}
\item[\rm(1)] $(\del\ssh)_X\colon\C(X,C)\to\E(X,A)\ :\ f\mapsto f\uas\del$.
\item[\rm (2)] $\del\ush_X\colon\C(A,X)\to\E(C,X)\ :\ g\mapsto g\sas\delta$.
\end{enumerate}
We simply denote $(\del\ssh)_X(f)$ and $\del\ush_X(g)$ by $\del\ssh(f)$ and $\del\ush(g)$, respectively.
\end{definition}

\begin{definition}\cite[Definition 2.9]{HLN}
 Let $\C,\E,n$ be as before. Define a category $\AE:=\AE^{n+2}_{(\C,\E)}$ as follows.
\begin{enumerate}
\item[\rm(1)]  A pair $\Xd$ is an object of the category $\AE$ with $X_{\bullet}\in\CC$
and $\del\in\E(X_{n+1},X_0)$, called an $\E$-attached
complex of length $n+2$, if it satisfies
$$(d_0^X)_{\ast}\del=0~~\textrm{and}~~(d^X_n)^{\ast}\del=0.$$
We also denote it by
$$X_0\xrightarrow{d_0^X}X_1\xrightarrow{d_1^X}\cdots\xrightarrow{d_{n-2}^X}X_{n-1}
\xrightarrow{d_{n-1}^X}X_n\xrightarrow{d_n^X}X_{n+1}\overset{\delta}{\dashrightarrow}.$$
\item[\rm (2)]  For such pairs $\Xd$ and $\langle Y_{\bullet},\rho\rangle$,  $f_{\bullet}\colon\Xd\to\langle Y_{\bullet},\rho\rangle$ is
defined to be a morphism in $\AE$ if it satisfies $(f_0)_{\ast}\del=(f_{n+1})^{\ast}\rho$.

\end{enumerate}
\end{definition}

\begin{definition}\cite[Definition 2.13]{HLN}\label{def1}
 An {\it $n$-exangle} is an object $\Xd$ in $\AE$ that satisfies the listed conditions.
\begin{enumerate}
\item[\rm (1)] The following sequence of functors $\C\op\to\Ab$ is exact.
$$
\C(-,X_0)\xrightarrow{\C(-,\ d^X_0)}\cdots\xrightarrow{\C(-,\ d^X_n)}\C(-,X_{n+1})\xrightarrow{~\del\ssh~}\E(-,X_0)
$$
\item[\rm (2)] The following sequence of functors $\C\to\Ab$ is exact.
$$
\C(X_{n+1},-)\xrightarrow{\C(d^X_n,\ -)}\cdots\xrightarrow{\C(d^X_0,\ -)}\C(X_0,-)\xrightarrow{~\del\ush~}\E(X_{n+1},-)
$$
\end{enumerate}
In particular any $n$-exangle is an object in $\AE$.
A {\it morphism of $n$-exangles} simply means a morphism in $\AE$. Thus $n$-exangles form a full subcategory of $\AE$.
\end{definition}

\begin{definition}\cite[Definition 2.22]{HLN}
Let $\s$ be a correspondence which associates a homotopic equivalence class $\s(\del)=[{}_A{X_{\bullet}}_C]$ to each extension $\del={}_A\del_C$. Such $\s$ is called a {\it realization} of $\E$ if it satisfies the following condition for any $\s(\del)=[X_{\bullet}]$ and any $\s(\rho)=[Y_{\bullet}]$.
\begin{itemize}
\item[{\rm (R0)}] For any morphism of extensions $(a,c)\co\del\to\rho$, there exists a morphism $f_{\bullet}\in\CC(X_{\bullet},Y_{\bullet})$ of the form $f_{\bullet}=(a,f_1,\ldots,f_n,c)$. Such $f_{\bullet}$ is called a {\it lift} of $(a,c)$.
\end{itemize}
In such a case, we simple say that \lq\lq$X_{\bullet}$ realizes $\del$" whenever they satisfy $\s(\del)=[X_{\bullet}]$.

Moreover, a realization $\s$ of $\E$ is said to be {\it exact} if it satisfies the following conditions.
\begin{itemize}
\item[{\rm (R1)}] For any $\s(\del)=[X_{\bullet}]$, the pair $\Xd$ is an $n$-exangle.
\item[{\rm (R2)}] For any $A\in\C$, the zero element ${}_A0_0=0\in\E(0,A)$ satisfies
\[ \s({}_A0_0)=[A\ov{\id_A}{\lra}A\to0\to\cdots\to0\to0]. \]
Dually, $\s({}_00_A)=[0\to0\to\cdots\to0\to A\ov{\id_A}{\lra}A]$ holds for any $A\in\C$.
\end{itemize}
Note that the above condition {\rm (R1)} does not depend on representatives of the class $[X_{\bullet}]$.
\end{definition}

\begin{definition}\cite[Definition 2.23]{HLN}
Let $\s$ be an exact realization of $\E$.
\begin{enumerate}
\item[\rm (1)] An $n$-exangle $\Xd$ is called an $\s$-{\it distinguished} $n$-exangle if it satisfies $\s(\del)=[X_{\bullet}]$. We often simply say {\it distinguished $n$-exangle} when $\s$ is clear from the context.
\item[\rm (2)]  An object $X_{\bullet}\in\CC$ is called an {\it $\s$-conflation} or simply a {\it conflation} if it realizes some extension $\del\in\E(X_{n+1},X_0)$.
\item[\rm (3)]  A morphism $f$ in $\C$ is called an {\it $\s$-inflation} or simply an {\it inflation} if it admits some conflation $X_{\bullet}\in\CC$ satisfying $d_0^X=f$.
\item[\rm (4)]  A morphism $g$ in $\C$ is called an {\it $\s$-deflation} or simply a {\it deflation} if it admits some conflation $X_{\bullet}\in\CC$ satisfying $d_n^X=g$.
\end{enumerate}
\end{definition}

\begin{definition}\cite[Definition 2.27]{HLN}
For a morphism $f_{\bullet}\in\CC(X_{\bullet},Y_{\bullet})$ satisfying $f_0=\id_A$ for some $A=X_0=Y_0$, its {\it mapping cone} $M_{_{\bullet}}^f\in\CC$ is defined to be the complex
\[ X_1\xrightarrow{d^{M_f}_0}X_2\oplus Y_1\xrightarrow{d^{M_f}_1}X_3\oplus Y_2\xrightarrow{d^{M_f}_2}\cdots\xrightarrow{d^{M_f}_{n-1}}X_{n+1}\oplus Y_n\xrightarrow{d^{M_f}_n}Y_{n+1} \]
where $d^{M_f}_0=\begin{bmatrix}-d^X_1\\ f_1\end{bmatrix},$
$d^{M_f}_i=\begin{bmatrix}-d^X_{i+1}&0\\ f_{i+1}&d^Y_i\end{bmatrix}\ (1\le i\le n-1),$
$d^{M_f}_n=\begin{bmatrix}f_{n+1}&d^Y_n\end{bmatrix}$.

{\it The mapping cocone} is defined dually, for morphisms $h_{\bullet}$ in $\CC$ satisfying $h_{n+1}=\id$.
\end{definition}

\begin{definition}\cite[Definition 2.32]{HLN}
An {\it $n$-exangulated category} is a triplet $(\C,\E,\s)$ of additive category $\C$, additive bifunctor $\E\co\C\op\times\C\to\Ab$, and its exact realization $\s$, satisfying the following conditions.
\begin{itemize}
\item[{\rm (EA1)}] Let $A\ov{f}{\lra}B\ov{g}{\lra}C$ be any sequence of morphisms in $\C$. If both $f$ and $g$ are inflations, then so is $g\circ f$. Dually, if $f$ and $g$ are deflations, then so is $g\circ f$.

\item[{\rm (EA2)}] For $\rho\in\E(D,A)$ and $c\in\C(C,D)$, let ${}_A\langle X_{\bullet},c\uas\rho\rangle_C$ and ${}_A\langle Y_{\bullet},\rho\rangle_D$ be distinguished $n$-exangles. Then $(\id_A,c)$ has a {\it good lift} $f_{\bullet}$, in the sense that its mapping cone gives a distinguished $n$-exangle $\langle M^f_{_{\bullet}},(d^X_0)\sas\rho\rangle$.
 \item[{\rm (EA2$\op$)}] Dual of {\rm (EA2)}.
\end{itemize}
Note that the case $n=1$, a triplet $(\C,\E,\s)$ is a  $1$-exangulated category if and only if it is an extriangulated category, see \cite[Proposition 4.3]{HLN}.
\end{definition}

\begin{example}
From \cite[Proposition 4.34]{HLN} and \cite[Proposition 4.5]{HLN},  we know that $n$-exact categories and $(n+2)$-angulated categories are $n$-exangulated categories.
There are some other examples of $n$-exangulated categories
 which are neither $n$-exact nor $(n+2)$-angulated, see \cite{HLN,LZ,HZZ2}.
\end{example}
The following some Lemmas are very useful which are needed later on.

\begin{lemma}\emph{\cite[Lemma 2.12]{LZ}}\label{a1}
Let $(\C,\E,\s)$ be an $n$-exangulated category, and
$$A_0\xrightarrow{\alpha_0}A_1\xrightarrow{\alpha_1}A_2\xrightarrow{\alpha_2}\cdots\xrightarrow{\alpha_{n-2}}A_{n-1}
\xrightarrow{\alpha_{n-1}}A_n\xrightarrow{\alpha_n}A_{n+1}\overset{\delta}{\dashrightarrow}$$
be a distinguished $n$-exangle. Then we have the following exact sequences:
$$\C(-, A_0)\xrightarrow{}\C(-, A_1)\xrightarrow{}\cdots\xrightarrow{}
\C(-, A_{n+1})\xrightarrow{}\E(-, A_{0})\xrightarrow{}\E(-, A_{1})\xrightarrow{}\E(-, A_{2});$$
$$\C(A_{n+1},-)\xrightarrow{}\C(A_{n},-)\xrightarrow{}\cdots\xrightarrow{}
\C(A_0,-)\xrightarrow{}\E(A_{n+1},-)\xrightarrow{}\E(A_{n},-)\xrightarrow{}\E(A_{n-1},-).$$
\end{lemma}

\begin{lemma}\emph{\cite[Proposition 3.6]{HLN}}\label{a2}
\rm Let ${}_A\langle X_{\bullet},\delta\rangle_C$ and ${}_B\langle Y_{\bullet},\rho\rangle_D$ be distinguished $n$-exangles. Suppose that we are given a commutative square
$$\xymatrix{
 X_0 \ar[r]^{{d_0^X}} \ar@{}[dr]|{\circlearrowright} \ar[d]_{a} & X_1 \ar[d]^{b}\\
 Y_0  \ar[r]_{d_0^Y} &Y_1
}
$$
in $\C$. Then there is a morphism $f_{\bullet}\colon \langle X_{\bullet},\delta\rangle\to\langle Y_{\bullet},\rho\rangle$ which satisfies $f_0=a$ and $f_1=b$.
\end{lemma}
\begin{lemma}\rm\cite[Lemma 2.11 ]{HHZZ}\label{y1}
Let $\C$ be an $n$-exangulated category , and $$\xymatrix{
X_0\ar[r]^{f_0}\ar@{}[dr] \ar[d]^{a_0} &X_1 \ar[r]^{f_1} \ar@{}[dr]\ar[d]^{a_1}\ar@{-->}[dl]^{h_1} &X_2 \ar[r]^{f_2} \ar@{}[dr]\ar[d]^{a_2}\ar@{-->}[dl]^{h_2}&\cdot\cdot\cdot \ar[r]\ar@{}[dr] &X_n \ar[r]^{f_n} \ar@{}[dr]\ar[d]^{a_n}&X_{n+1} \ar@{}[dr]\ar[d]^{a_{n+1}} \ar@{-->}[dl]^{h_{n+1}}\ar@{-->}[r]^-{\delta} &\\
{Y_0}\ar[r]^{g_0} &{Y_1}\ar[r]^{g_1}&{Y_2} \ar[r]^{g_2} &\cdot\cdot\cdot \ar[r] &{Y _n}\ar[r]^{g_n}  &{Y_{n+1}} \ar@{-->}[r]^-{\eta} &}
$$
any morphism of distinguished $n$-exangles. Then the following are equivalent:
\begin{itemize}
\item[\rm (1)]There is a morphism $h_1\colon X_1\to Y_0$, such that $h_1f_0=a_0$.

\item[\rm (2)]There is a morphism $h_{n+1}\colon X_{n+1}\to Y_n$, such that $g_nh_{n+1}=a_{n+1}$.

\item[\rm (3)] $ (a_0)_{*}{\delta}=(a_{n+1})^{*}{\eta}=0$.

\item[\rm (4)] $a_{\bullet}=(a_0,a_1,\cdot\cdot\cdot,a_{n+1})\colon\Xd\to\langle Y_{\bullet},\eta\rangle$ is null-homotopic.

\end{itemize}
\end{lemma}

\begin{corollary}\rm\cite[Corollary 2.12 ]{HHZZ}\label{y2}
If $a_{\bullet}$ is the identity on $\Xd$ as above,  then the following are equivalent:
\begin{itemize}
\item[\rm (1)] $f_0$ is a split monomorphism (also known as a section).

\item[\rm (2)] $f_n$ is a split epimorphism (also known as a retraction).

\item[\rm (3)] $ {\delta}=0$.

\item[\rm (4)] $a_{\bullet}$ is null-homotopic.

\end{itemize}
If a distinguished $n$-exangle satisfies one of the above equivalent conditions, it is called \emph{split}.
\end{corollary}

We denote by ${\rm rad}_{\C}$ the Jacobson radical of $\C$. Namely, ${\rm rad}_{\C}$ is an ideal of $\C$ such that ${\rm rad}_{\C}(A, A)$
coincides with the Jacobson radical of the endomorphism ring ${\rm End}(A)$ for any $A\in\C$.

\begin{definition}\cite[Definition 3.3 ]{HZ} When $n\geq2$, a distinguished $n$-exangle in $\C$ of the form
$$A_{\bullet}:~~A_0\xrightarrow{\alpha_0}A_1\xrightarrow{\alpha_1}A_2\xrightarrow{\alpha_2}\cdots\xrightarrow{\alpha_{n-2}}A_{n-1}
\xrightarrow{\alpha_{n-1}}A_n\xrightarrow{\alpha_n}A_{n+1}\overset{}{\dashrightarrow}$$
is minimal if $\alpha_1,\alpha_2,\cdots,\alpha_{n-1}$ are in $\rad_{\C}$.

\end{definition}
The following lemma shows that a distinguished $n$-exangle in an equivalence class can be chosen in a minimal way in a Krull-Schmidt $n$-exangulated category.

\begin{lemma}\rm\label{ml}\cite[Lemma 3.4 ]{HZ} Let $\C$ be a Krull-Schmidt $n$-exangulated category, $A_0,A_{n+1}\in\C$. Then for every equivalence class associated with $\E$-extension $\del={}_{A_0}\del_{A_{n+1}}$, there exists a representation
$$A_{\bullet}:~~A_0\xrightarrow{\alpha_0}A_1\xrightarrow{\alpha_1}A_2\xrightarrow{\alpha_2}\cdots\xrightarrow{\alpha_{n-2}}A_{n-1}
\xrightarrow{\alpha_{n-1}}A_n\xrightarrow{\alpha_n}A_{n+1}\overset{\delta}{\dashrightarrow}$$
such that $\alpha_1,\alpha_2,\cdots,\alpha_{n-1}$ are in $\rad_{\C}$. Moreover, $A_{\bullet}$ is a direct summand of every other elements in this equivalent class.

\end{lemma}

\section{$n$-exangulated categories having Auslander-Reiten-Serre duality (Serre duality)}
Let $k$ be a commutative artinian ring. In the rest of this paper, we always assume that $(\C,\E,\s)$ is an $\Ext$-finite, Krull-Schmidt and $k$-linear $n$-exangulated category. Here, an $n$-exangulated category $(\C,\E,\s)$ is $k$-linear if $\C(A,B)$ and $\E(A,B)$ are $k$-modules such that the following compositions
 \begin{eqnarray*}
&\C(A,B)\times \C(B,C)\to \C(A,C),&\\
&\C(A,B)\times \E(B,C)\times  \C(C,D)\to \E(A,D)&\end{eqnarray*}
are $k$-linear for any $A,B,C,D\in\C$, and is $\Ext$-finite if $\E(A,B)$ is a finitely generated $k$-module for any $A,B \in\C$.

\subsection{$n$-exangulated categories having Auslander-Reiten-Serre duality}
\begin{definition}\cite[Definition 3.1]{HHZZ}\label{222} Let $\C$ be an $n$-exangulated category. A distinguished $n$-exangle
$$A_0\xrightarrow{\alpha_0}A_1\xrightarrow{\alpha_1}A_2\xrightarrow{\alpha_2}\cdots\xrightarrow{\alpha_{n-2}}A_{n-1}
\xrightarrow{\alpha_{n-1}}A_n\xrightarrow{\alpha_n}A_{n+1}\overset{\delta}{\dashrightarrow}$$
in $\C$ is called an \emph{Auslander-Reiten $n$-exangle }if
$\alpha_0$ is left almost split, $\alpha_n$ is right almost split and
when $n\geq 2$, $\alpha_1,\alpha_2,\cdots,\alpha_{n-1}$ are in $\rad_{\C}$.
\end{definition}

\begin{lemma}\rm\cite[Lemma 3.3]{HHZZ}\label{r1}
Let $\C$ be an $n$-exangulated category and
$$A_{\bullet}:~~A_0\xrightarrow{\alpha_0}A_1\xrightarrow{\alpha_1}A_2\xrightarrow{\alpha_2}\cdots\xrightarrow{\alpha_{n-2}}A_{n-1}
\xrightarrow{\alpha_{n-1}}A_n\xrightarrow{\alpha_n}A_{n+1}\overset{\delta}{\dashrightarrow}$$
be a distinguished $n$-exangle in $\C$. Then the following statements are equivalent:
\begin{itemize}
\item[\rm (1)] $A_{\bullet}$ is an Auslander-Reiten $n$-exangle;
\item[\rm (2)] ${\rm{End}}(A_0)$ is local, if $n\geq 2$, $ \alpha_1,\cdots,\alpha_{n-1}$ are in ${\rm rad}_{\C}$ and $\alpha_n$ is right almost split;
\item[\rm (3)] ${\rm{End}}(A_{n+1})$ is local, if $n\geq 2$, $\alpha_1,\alpha_2,\cdots,\alpha_{n-1}$ are in ${\rm rad}_{\C}$ and $\alpha_0$ is left almost split.
\end{itemize}
\end{lemma}

\begin{definition}
We say that \emph{$\C$ has right {Auslander-Reiten $n$-exangle}} if for any indecomposable non-projective object $A\in\C$, there exists an {Auslander-Reiten $n$-exangle} ending at $A$.
Dually, we say that \emph{$\C$ has left {Auslander-Reiten $n$-exangle}} if for any indecomposable non-injective object $B\in\C$, there exists an {Auslander-Reiten $n$-exangle} starting at $B$.
We say that \emph{$\mathscr{C}$ has Auslander-Reiten $n$-exangles} if it has right and left Auslander-Reiten $n$-exangles.
\end{definition}

\begin{definition}\rm\cite[Definition 3.1]{HHZ}\label{De1}
\rm (1) A morphism $f:X \rightarrow Y$ in $\C$ is called \emph{n-projectively trivial} if for each $Z\in\C$, the induced map $\E(f,Z):\E(Y,Z)\rightarrow\E(X,Z)$ is zero. Dually, a morphism $g:X \rightarrow Y$ in $\C$ is called \emph{n-injectively trivial} if for each $Z\in\C$, the induced map $\E(Z,g):\E(Z,X)\rightarrow\E(Z,Y)$ is zero.

\rm (2) Let $C\in\C$. We call $C$ a \emph{projective object} if the identity morphism $\Id_C$ is \emph{n-projectively trivial}, and an \emph{injective object} if the identity morphism $\Id_C$ is \emph{n-injectively trivial}.
\end{definition}

We introduce some concepts, which will be used later.

$\bullet$ Let $A$ and $B$ be two objects in $\C$. We denote by ${\P}(A,B)$ the set of $n$-projectively trivial morphisms from $A$ to $B$.
The \emph{stable category} $\underline{\C}$ of $\C$ is defined as follows, the category whose objects are objects of $\C$ and whose morphisms are elements of
${\underline{\C}}(A,B)={\C}(A,B)/\P(A,B)$. Given a morphism $f\colon A\to B$ in $\C$, we denote by $\underline{f}$ the image of $f$ in $\underline{\C}$. Dually, We denote by ${\I}(A,B)$ the set of $n$-injectively trivial morphisms from $A$ to $B$.
The \emph{costable category} $\overline{\C}$ of $\C$ is defined dually. Given a morphism $g\colon A\to B$ in $\C$, we denote by $\overline{g}$ the image of $g$ in $\overline{\C}$.

We denote by ${\rm mod} k$ the category of finitely generated $k$-modules. Let $E$ be the minimal injective cogenerator of $k$. Then we have the duality $D=\Hom_k(-, E)$.

$\bullet$ The category $\C$ is said to have \emph{Auslander-Reiten-Serre duality} provided that there exists a $k$-linear equivalence $\tau_{n}:\underline{\C}\rightarrow\overline{\C}$ with a $k$-linear natural isomorphism $$\Phi_{X,Y}: D\E(X,Y)\rightarrow\overline{{\C}}(Y,\tau_{n} X)$$
for any $X,Y\in\C$. The equivalence $\tau_{n}$ is called the \emph{Auslander-Reiten-translation} of $\C$.

$\bullet$ We denote by $\tau_{n}^{-}$ a quasi-inverse of $\tau_{n}$. It is well known that the pair $(\tau_{n}^{-},\tau_{n})$ is an adjoint pair. We denote by the counit $\theta:\tau_{n}^{-}\tau_{n}\rightarrow \Id_{\underline{{\C}}} $ and the unit $\epsilon:\Id_{\overline{{\C}}}\rightarrow \tau_{n}\tau_{n}^{-} $. For any $X,Y\in\C$, there is an isomorphism
$$\vartheta_{X,Y}:\overline{{\C}}(Y,\tau_{n} X)\rightarrow \underline{{\C}}(\tau_{n}^{-}Y, X),~~~f\mapsto \theta_{X}\tau_{n}^{-}(f).$$

$\bullet$ For any $X,Y\in\C$, there  exists a natural isomorphism $$\Psi_{X,Y}:D\E(X,Y)\rightarrow \underline{{\C}}(\tau_{n}^{-}Y, X),~~~f\mapsto \vartheta_{\tau_{n}^{-}X,Y}(\Phi_{\tau_{n}^{-}X,Y}(f))$$     by the composition of $\Phi_{X,Y}$ and $\vartheta_{X,Y}$.

Now we are ready to state and prove our first main result.
\begin{theorem}\label{theorem1}
Let $\C$ be an $\Ext$-finite, Krull-Schmidt and $k$-linear $n$-exangulated category. Then the following conditions are equivalent.
\begin{enumerate}
\item[$(1)$] $\C$ has Auslander-Reiten $n$-exangles.
\item[$(2)$] $\C$ has an Auslander-Reiten-Serre duality.

\end{enumerate}
 \end{theorem}

{\bf In order to prove Theorem \ref{theorem1}, we need some preparations as follows.}

\begin{lemma}\label{01}
For any non-split $\delta\in\E(X_{n+1},X_0)$ with $\s(\delta)=[X_0\xrightarrow{\alpha_0}X_1\xrightarrow{\alpha_1}X_2\xrightarrow{\alpha_2}\cdots\xrightarrow{\alpha_{n-2}}X_{n-1}
\xrightarrow{\alpha_{n-1}}X_n\xrightarrow{\alpha_{n}}X_{n+1}]$.
\item[\rm (1)]If $\alpha_0$ is left almost split morphism, then the following holds for any $Y_{n+1}\in\C$.
\begin{itemize}
\item[\rm (a)]For any $0\neq\eta\in\E(Y_{n+1},X_0)$, there exists $\varphi_{n+1}\in\C(X_{n+1},Y_{n+1})$ such that $\delta=\eta\varphi_{n+1}$.
\item[\rm (b)]For any $0\neq \overline {a}\in\overline {\C}(Y_{n+1},X_0)$, there exists $\gamma\in\E(X_{n+1},Y_{n+1})$ such that $\delta=a\gamma$.

\end{itemize}
\item[\rm (2)]If $\alpha_n$ is right almost split morphism, then the following holds for any $Y_0\in\C$.
\begin{itemize}
\item[\rm (c)]For any $0\neq\gamma\in\E(X_{n+1},Y_0)$, there exists $\varphi_0\in\C(Y_0,X_0)$ such that $\delta=\varphi_0\gamma$.
\item[\rm (d)]For any $0\neq \underline {a}\in\underline {\C}(X_{n+1},Y_0)$, there exists $\eta\in\E(Y_0,X_0)$ such that $\delta=\eta a$.

\end{itemize}
\end{lemma}

\begin{proof}
We only prove that $(1)$, dually one can prove $(2)$. \rm (a) Consider the following two distinguished $n$-exangle
$$X_0\xrightarrow{\alpha_0}X_1\xrightarrow{\alpha_1}X_2\xrightarrow{\alpha_2}\cdots\xrightarrow{\alpha_{n-2}}X_{n-1}
\xrightarrow{\alpha_{n-1}}X_n\xrightarrow{\alpha_{n}}X_{n+1}\overset{\delta}{\dashrightarrow},$$ $$X_0\xrightarrow{\beta_0}Y_1\xrightarrow{\beta_1}Y_2\xrightarrow{\beta_2}\cdots\xrightarrow{\beta_{n-2}}Y_{n-1}
\xrightarrow{\beta_{n-1}}Y_n\xrightarrow{\beta_{n}}Y_{n+1}\overset{\eta}{\dashrightarrow}.$$
Since $\eta\neq0$, we have that $\beta_0$ is not split monomorphism. Note that $\alpha_0$ is left almost split morphism, thus there is some $\varphi_1\in\C(X_1,Y_1)$ satisfying $\varphi_1\alpha_0=\beta_0$. By Lemma \ref{a2}, we have the
following commutative diagram of distinguished $n$-exangles
$$\xymatrix{X_0\ar[r]^{\alpha_0}\ar@{}[dr] \ar@{=}[d]^{} &X_1 \ar[r]^{\alpha_1} \ar@{}[dr]\ar[d]^{\varphi_1}&\cdot\cdot\cdot \ar[r]^{\alpha_{n-2}} \ar@{}[dr]&X_{n-1} \ar[r]^{\alpha_{n-1}}\ar@{}[dr]\ar@{-->}[d]^{\varphi_{n-1}} &X_n \ar[r]^{\alpha_{n}} \ar@{}[dr]\ar@{-->}[d]^{\varphi_n}&X_{n+1} \ar@{}[dr]\ar@{-->}[d]^{\varphi_{n+1}} \ar@{-->}[r]^-{\delta} &\\
{X_0}\ar[r]^{\beta_0} &{Y_1}\ar[r]^{\beta_1}&\cdot\cdot\cdot\ar[r]^{\beta_{n-2}} &{Y_{n-1}}  \ar[r]^{\beta_{n-1}} &{Y _n}\ar[r]^{\beta_n}  &{Y_{n+1}} \ar@{-->}[r]^-{\eta} &.}
$$
In particular it satisfies $\delta=\eta\varphi_{n+1}$.

\rm (b) Suppose that $ {a}\in\overline {\C}(Y_{n+1},X_0)$ does not belong to $\I$. Then there exist $Y\in\C$, such that the map $a_{*}:\E(Y,Y_{n+1})\rightarrow\E(Y,X_0)$ is non-zero. We can take $\zeta\in\E(Y,Y_{n+1})$ such that $a\zeta\neq 0$. By \rm (a), there exists a morphism  $c\in\C(X_{n+1},Y)$ such that $\delta=(a\zeta)c$. We have $\gamma=\zeta c\in\E(X_{n+1},Y_{n+1})$ as desire.
\end{proof}

\begin{lemma}\label{02} Let $X_0$ be a non-injective indecomposable object and $X_{n+1}$ a non-projective indecomposable object in $\C$. Then the following statements are equivalent.
\begin{enumerate}
\item[$(1)$] There exists an {Auslander-Reiten $n$-exangle} of the form
$$X_{\bullet}:X_0\xrightarrow{\alpha_{0}}X_1\xrightarrow{\alpha_{1}}X_2\xrightarrow{}\cdots\xrightarrow{\alpha_{n-1}}X_n\xrightarrow{\alpha_{n}}X_{n+1}\overset{\delta}{\dashrightarrow}.$$
\item[$(2)$] There exists an isomorphism $\underline{\C}(X_{n+1},-)\cong D\E(-,X_0)$ of functors on $\C$.
\item[$(3)$] There exists an isomorphism $\E(X_{n+1},-)\cong D\overline{\C}(-,X_0)$ of functors on $\C$.
\end{enumerate}
\end{lemma}
\begin{proof}
$(1)\Rightarrow (2)$ and $(1)\Rightarrow (3)$ Since $X_{\bullet}$ is an {Auslander-Reiten $n$-exangle}, then $\alpha_{0}$ is left almost split, $\alpha_{n}$ is right almost split and $\delta\neq0$. Take any linear form $\eta:\E(X_{n+1},X_0)\rightarrow k$ satisfying $\eta(\delta)\neq 0$. By Lemma \ref{01}, we have two $k$-bilinear forms
$$\underline{\C}(X_{n+1},X)\times\E(X,X_0)\xrightarrow{} \E(X_{n+1},X_0)\xrightarrow{\eta} k ,~~~(\underline{a},\gamma)\mapsto\eta(\gamma a),$$
$$\E(X_{n+1},X)\times\overline{\C}(X,X_0)\xrightarrow{} \E(X_{n+1},X_0)\xrightarrow{\eta} k ,~~~(\gamma,\overline{b})\mapsto\eta(b\gamma)$$
are \emph{non-degenerated} for each $X\in\C$. Since $\C$ is $\Ext$-finite, the above two $k$-bilinear forms induce two natural isomorphism $\underline{\C}(X_{n+1},-)\cong D\E(-,X_0)$ and $\E(X_{n+1},-)\cong D\overline{\C}(-,X_0)$.

(2)$\Rightarrow$(1) For the object $X_{n+1}$, there exists an isomorphism $\End_{\underline{\C}}(X_{n+1})\cong D\E(X_{n+1},X_0)$. Since $\End(X_{n+1})$ is local, we know that $\Xi_{X_{n+1}}:=\End_{\underline{\C}}(X_{n+1})/\rad\End_{\underline{\C}}(X_{n+1})$ is a simple $\End_{\mathscr{C}}(X_{n+1})$-module. Then $\Gamma_{X_{n+1}}:=D(\Xi_{X_{n+1}})$ is a simple $\End_{\mathscr{C}}(X_{n+1})$-submodule of $\E(X_{n+1},X_0)$. For $ 0\neq\delta\in\Gamma_{X_{n+1}}$, consider the distinguished $n$-exangle
$$X_{\bullet}: X_0\xrightarrow{\alpha_0}X_1\xrightarrow{\alpha_1}X_2\xrightarrow{\alpha_2}\cdots\xrightarrow{\alpha_{n-2}}X_{n-1}
\xrightarrow{\alpha_{n-1}}X_n\xrightarrow{\alpha_{n}}X_{n+1}\overset{\delta}{\dashrightarrow},$$
where we may assume $\alpha_{i}\in\rad_{\C}$ for every $1\leq{i}\leq n-1$ by Lemma \ref{ml}. Next we show that $\alpha_{n}$ is a right almost split morphism. Since $0\neq\delta$, we have $\alpha_{n}$ is not a split epimorphism. Suppose that $\beta:B\rightarrow X_{n+1}$ is not a split epimorphism in $\C$, then the composition $${\underline{\C}}(X_{n+1},B)\xrightarrow{\underline{\beta}\circ-}{\underline{\C}}(X_{n+1},X_{n+1})\to \Xi_{X_{n+1}}~~~~~~~~~~~(\clubsuit)$$ is zero. We claim that $\beta^\ast\delta=0$. The isomorphism $\iota:\E(-,X_0)\cong D\underline{\C}(X_{n+1},-)$ of functors on $\C$ gives the following commutative diagram
\[\xymatrix{
  \E(X_{n+1} ,X_0)\ar[r]^-{\beta^\ast}\ar[d]_{\iota_{X_{n+1}}}\ar@{}[dr]|{\spadesuit}
    &\E(B ,X_0)\ar[d]^-{\iota_{B}}\\
  D\underline\C(X_{n+1},X_{n+1})\ar[r]^-{D(\underline{\beta}\circ-)}
    &D\underline\C(X_{n+1} ,B).
}\]
Note that the $(\clubsuit)$, we have the composition $$D(\Xi_{X_{n+1}})\xrightarrow{} D{\underline{\C}}(X_{n+1},X_{n+1})\xrightarrow{D(\underline{\beta}\circ-)}D{\underline{\C}}(X_{n+1},B)$$ is zero. Hence $\beta^\ast\delta=0$ by the commutativity of the square $\spadesuit$. Consider the following commutative diagram
$$\xymatrix{
X_0\ar[r]^{\beta_0}\ar@{}[dr] \ar@{=}[d]^{} &Y_1 \ar[r]^{\beta_1} \ar@{}[dr]\ar[d]^{\psi_1}\ar@{-->}[dl]^{h_1} &Y_2 \ar[r]^{\beta_2} \ar@{}[dr]\ar[d]^{}&\cdot\cdot\cdot \ar[r]\ar@{}[dr] &Y_n \ar[r]^{\beta_n} \ar@{}[dr]\ar[d]^{\psi_n}&B \ar@{}[dr]\ar[d]^{\beta} \ar@{-->}[dl]^{h_{n+1}}\ar@{-->}[r]^-{\beta^\ast\delta} &\\
{X_0}\ar[r]^{\alpha_0} &{X_1}\ar[r]^{\alpha_1}&{X_2} \ar[r]^{\alpha_2} &\cdot\cdot\cdot \ar[r] &{X_n}\ar[r]^{\alpha_n}  &{X_{n+1}} \ar@{-->}[r]^-{\delta} &.}
$$
Since $\beta^\ast\delta=0$, then $\beta_0$ is split monomorphism, there is a morphism $h_1\colon Y_1\to X_0$, such that $h_1\beta_0=\id_{X_0}$. So $\beta$ factors through $\alpha_n$ by Lemma \ref{y1}. That is, $\alpha_{n}$ is a right almost split morphism.
By Lemma \ref{r1}, note that $\End(X_0)$ is local, we know that $$X_{\bullet}: X_0\xrightarrow{\alpha_0}X_1\xrightarrow{\alpha_1}X_2\xrightarrow{\alpha_2}\cdots\xrightarrow{\alpha_{n-2}}X_{n-1}
\xrightarrow{\alpha_{n-1}}X_n\xrightarrow{\alpha_{n}}X_{n+1}\overset{\delta}{\dashrightarrow}$$
is an Auslander-Reiten $n$-exangle in $\C$.

(3)$\Rightarrow$(1) is similar to (2)$\Rightarrow$(1).
\end{proof}

\begin{definition}\rm\cite[Lemma 3.8]{INP}
Let $(\C,\E,\D)$ be a triple consisting of $k$-linear additive categories $\C$ and $\D$ and a $k$-linear bifunctor $\E\colon\C^\mathrm{op}\times\D\to mod k$.
A \emph{right Auslander-Reiten-Serre duality} for $(\C,\E,\D)$ is a pair $(F,\eta)$ of a $k$-linear functor $F\colon\C\to \D$ and a binatural isomorphism
\[\eta_{A,B}\colon\C(A,B)\simeq D\E(B,FA)\ \mbox{ for any }\ A,B\in\C.\]
If moreover $F$ is an equivalence, we say that $(F,\eta)$ is an \emph{Auslander-Reiten-Serre duality} for $(\C,\E,\D)$.

Dually we define a \emph{left Auslander-Reiten-Serre duality} for $(\C,\E,\D)$.
\end{definition}
The following lemmas hold in any $k$-linear additive categories $\C$ and $\D$.

\begin{lemma}\rm\cite[Lemma 3.9]{INP}\label{rr1}
If $(F,\eta)$ is an Auslander-Reiten-Serre duality for $(\C,\E,\D)$, then $(G,\zeta)$ is a left Auslander-Reiten-Serre duality for $(\C,\E,\D)$, where $G$ is a quasi-inverse of $F$ and $\zeta_{A,B}$ is a composition
\[\D(A,B)\xrightarrow{G}\C(GA,GB)\xrightarrow{\eta_{GA,GB}}D\E(GB,FGA)\simeq D\E(GB,A)\]
for any $A,B\in\D$.
\end{lemma}

\begin{lemma}\rm\cite[Lemma 3.10]{INP}\label{r1}
Let $(\C,\E,\D)$ be a triple consisting of $k$-linear additive categories $\C$ and $\D$, and a $k$-linear bifunctor $\E\colon\C^\mathrm{op}\times\D\to{\rm mod} k$. Assume that we have the following.
\begin{enumerate}
\item[$\bullet$] A correspondence $F$ from objects in $\C$ to objects in $\D$.
\item[$\bullet$] A $k$-linear map $\eta_A\colon E(A,FA)\to k$ for any $A\in\C$ such that the compositions
\begin{eqnarray*}\label{pairing}
&\mathscr{C}(A,B)\times \E(B,FA)\to \E(A,FA)\xrightarrow{\eta_A}k,&\\ \label{pairing 2}
&\E(B,FA)\times\D(FA,FB)\to \E(B,FB)\xrightarrow{\eta_B}k&
\end{eqnarray*}
are non-degenerate for any $A,B\in\C$.
\end{enumerate}
Then we can extend $F$ to a fully faithful functor $F\colon\C\to\D$ such that the pair $(F,\eta)$ is a right Auslander-Reiten-Serre duality for $(\C,\E,\D)$, where $\eta_{A,B}(f)(\delta)=\eta_A(\delta f)$.
\end{lemma}

{\bf Now we are ready to prove Theorem \ref{theorem1}.}
 \begin{proof}
\textbf{Step 1:} First of all, we show that $\C$ has right Auslander-Reiten $n$-exangles if and only if $\C$ has a right Auslander-Reiten-Serre duality $(\tau_{n},\eta)$ with $\tau_{n}$ is fully faithful.

$``\Leftarrow "$ It follows from Lemma \ref{02}.

$``\Rightarrow "$ Let $A$ be an indecomposable non-projective object, we fix some object $FA$ such that $\underline{\C}(A,-)\cong D\E(-,FA)$. Suppose that $$ FA\xrightarrow{}X_1\xrightarrow{}X_2\xrightarrow{}\cdots\xrightarrow{}X_{n-1}
\xrightarrow{}X_n\xrightarrow{}A\overset{\delta_A}{\dashrightarrow}$$
is an Auslander-Reiten $n$-exangle for some $\delta\in\E(A,FA)$. Take any linear form $\eta_A:\E(A,FA)\rightarrow k$ satisfying $\eta_A(\delta_A)\neq 0$. By Lemma \ref{01}, we have two $k$-bilinear forms
$$\underline{\C}(A,-)\times\E(-,FA)\xrightarrow{} \E(A,FA)\xrightarrow{\eta_A} k$$
$$\E(A,-)\times\overline{\C}(-,FA)\xrightarrow{} \E(A,FA)\xrightarrow{\eta_A} k $$
are \emph{non-degenerated}. We can extend this to any object in $\underline{\C}$. Applying Lemma \ref{r1} to $({\underline{\C}},\E,{\overline{\C}})$, we have a right Auslander-Reiten-Serre duality $(F,\eta)$ such that $F:\underline{\C}\rightarrow \overline{\C}$ is fully faithful.

\textbf{Step 2:} (2)$\Rightarrow$(1) Suppose that $\C$ has an Auslander-Reiten-Serre duality $(\tau_{n},\eta)$. In particular, this is a right Auslander-Reiten-Serre duality. Then $\C$ has right Auslander-Reiten $n$-exangles by step 1. By Lemma \ref{rr1}, $\C$ has left Auslander-Reiten-Serre duality. Hence $\C$ has left Auslander-Reiten $n$-exangles by the dual of  step 1. this shows that $\C$ has Auslander-Reiten $n$-exangles.

(1)$\Rightarrow$(2) By step 1, $\C$ has a right Auslander-Reiten-Serre duality $(\tau_{n},\eta)$ and $\tau_{n}:\underline{\C}\rightarrow\overline{\C}$ is fully faithful. We only need to show $\tau_{n}$ is dense. This follows our assumption that $\C$ has left Auslander-Reiten $n$-exangles since $\tau_{n}$ sends the right term of an Auslander-Reiten $n$-exangle to its left term.

\end{proof}

\begin{remark}In Theorem \ref{theorem1}, when $\C$ is a triangulated category, it is just Theorem I.6.3 in
\cite{RV}, when $\C$ is an extriangulated category, it is just Theorem 3.6 in
\cite{INP}, when $\C$ is an $(n+2)$-angulated category, it is just Theorem 3.3 in
\cite{Z}. when $\C$ is an $n$-abelian category with enough projectives and enough injectives, it is a new phenomena.
\end{remark}

\subsection{$n$-exangulated categories having Serre duality}

~~~~$\bullet$ The category $\C$ is said to have \emph{Serre duality} (which is a special type of
Auslander-Reiten-Serre duality) provided that there exists a $k$-linear auto-equivalence $\tau_{n}:{\C}\rightarrow{\C}$ with a natural isomorphism $\varphi_{X,Y}: D\E(X,Y)\rightarrow{\C}(Y,\tau_{n} X)$ for any $X,Y\in\C$.

$\bullet$ Assume that $\C$ has Serre duality. For any projective object $P$, we have ${\C}(\tau_{n} P,\tau_{n} P)\cong D\E(P,\tau_{n} P)=0$, which implies $\tau_{n} P=0$. Thus $\tau_{n}$  induces a functor $\tau_{n}:\underline{\C}\rightarrow{\C}$. Similarly, $\tau_{n}^{-}$  induces a functor $\tau_{n}^{-}:\overline{\C}\rightarrow{\C}$.

\begin{definition}(Auslander \cite{Au})
Let $f\in \C(X, Y)$ and $C\in\C$. The morphism $f$ is called \emph{right $C$-determined} and $C$ is called a \emph{right determiner} of $f$, if the following condition is satisfied: each
$g\in \C(L, Y)$ factors through $f$, provided that for each $h\in\C(C, L)$ the morphism $g\circ h$ factors through $f$.
\end{definition}

\begin{definition}An object $Y\in\C$ is right deflation-classified provided that the following hold.

 (RDC1) Each deflation $\alpha:X\rightarrow Y$ ending at $Y$ is right $C$-determined for some $C\in\C$.

(RDC2) For any $C\in\C$ and ${\rm{End}}_{\C}(C)^\mathrm{op}$-submodule $H$ of ${\C}(C,Y)$, there exists a deflation $\alpha:X\rightarrow Y$ ending at $Y$ such that $\alpha$ is right $C$-determined and $\Im~{\C}(C,\alpha)=H$.

\end{definition}

\begin{remark}\label{re1}\rm (1) $\C$ is said to have right determined deflations if each object in $\C$ is right deflation-classified. Dually, one can define left inflation-classified objects and $\C$ having left determined inflations.

\rm (2) For any $n$-projectively trivial morphism $f:Z\rightarrow Y$, we have $f$ factors through any deflation $\alpha:X\rightarrow Y$ by Lemma 3.2 in \cite{HHZ}. In particular, Assume that $Y\in\C$ is right deflation-classified. For any $Z\in\C$, taking $H=0$, then there exists a deflation $\alpha:X\rightarrow Y$ such that $\alpha$ is right $Z$-determined and $\Im~{\C}(Z,\alpha)=0$ by (RDC2). Therefore, if $\C$ has right determined deflations, then $\P=\{0\}$ and $ \C=\underline{\C}$. Dually, if $\C$ has left determined inflations, then $\I=\{0\}$ and $ \C=\overline{\C}$.

\end{remark}

Next we are ready to state and prove our second main result.
\begin{theorem}\label{theorem2}
The following statements are equivalent.
\begin{enumerate}
\item[$(1)$] $\C$ has Serre duality.
\item[$(2)$] $\C$ has right determined deflations and left determined inflations.

\end{enumerate}
 \end{theorem}

{\bf In order to prove Theorem \ref{theorem2}, we need some preparations as follows.}

$\bullet$ Let $Y$ be an object in $\C$ and $H$ any ${\rm{End}}_{\C}(C)^\mathrm{op}$-submodule of ${\C}(C,Y)$ with $C\in{\C}$. since $D\C(C,C)$ is an injective cogenerator, there exist an embedding
$\C(C, Y )/H \hookrightarrow D\C(C^{'}, C)$
with $C^{'}\in \add C$. Hence we have a morphism $\overline {\varrho}:{\C}(-,Y)\rightarrow D\C(C^{'}, -)$. Take $\Im\overline {\varrho}=F^{(C,H)}$.

\begin{lemma}\rm\cite[Lemma 2.4]{Ch}\label{le0} Assume that $\C$ is a Hom-finite R-linear additive category. Let $H$ be an ${\rm{End}}_{\C}(C)^\mathrm{op}$-submodule of ${\C}(C,Y)$. Then $\alpha:X\rightarrow Y$ is right $C$-determined and $\Im\C(C, \alpha) =H$ if and only if the functor $F^{(C,H)}$ is finitely presented.
\end{lemma}

\begin{proposition}\label{pro1}
Let $Y\in\C$ be right deflation-classified. If $Y$ is indecomposable and non-projective, then there exists an {Auslander-Reiten $n$-exangle} ending at $Y$.

\end{proposition}

\begin{proof}Assume that $ H=\rad~ {\rm End}_{\C}(Y)$. We know that there is a deflation $f:X_n\rightarrow Y$ such that $f$ is right $Y$-determined and $\Im~{\C}(Y,f)=H $ by {\rm (RDC2)}. Without loss of generality, we may assume that $f$ is right minimal. Consider the distinguished $n$-exangle
$$X_{\bullet}:~~X_0\xrightarrow{f_0}X_1\xrightarrow{f_1}X_2\xrightarrow{f_2}\cdots\xrightarrow{f_{n-2}}X_{n-1}
\xrightarrow{f_{n-1}}X_n\xrightarrow{f}Y\overset{\eta}{\dashrightarrow},$$
where we may assume $f_{i}\in\rad_{\C}$ for every $0\leq{i}\leq n-1$ by Lemma 4.11 in \cite{HHZ}. Next we want to prove that $X_{\bullet}$ is an {Auslander-Reiten $n$-exangle}.

\textbf{Step 1:} We claim that $f$ is right almost split. For any non-split epimorphism $h\in{\C}(X',Y)$, it is clear that $hg$ is non-split epimorphism for any $g\in {\C}(Y,X')$. That is, $hg\in H$ since $Y$ is indecomposable. Note that $\Im~{\C}(Y,f)=H $, then we have $hg$ factors through $f$. Moreover, since $f$ is right $Y$-determined, then $h$ factors through $f$.

\textbf{Step 2:} We prove that $X_0$ is indecomposable. Assume that $X_0=\bigoplus\limits_{i=1}^m K_i$, where $K_i$ is indecomposable, $i=1,2,\cdots,m$. Since $f_0$ is non-split monomorphism, there exists some $K_i$ with $1\leq i\leq m$ such that the natural projection $\pi_i:X_0\rightarrow K_i$ does not factor through $f_0$. We have the following commutative diagram of distinguished $n$-exangles by {\rm (R0)}
$$\xymatrix{X_0\ar[r]^{f_0}\ar@{}[dr] \ar[d]^{\pi_i} &X_1 \ar[r]^{f_1} \ar@{}[dr]\ar@{-->}[d]^{\varphi_1}&\cdot\cdot\cdot \ar[r]^{f_{n-2}} \ar@{}[dr]&X_{n-1} \ar[r]^{f_{n-1}}\ar@{}[dr]\ar@{-->}[d]^{\varphi_{n-1}} &X_n \ar[r]^{f} \ar@{}[dr]\ar[d]^{\varphi_n}&Y \ar@{}[dr]\ar@{=}[d]^{} \ar@{-->}[r]^-{\eta} &\\
{K_i}\ar[r]^{g_0} &{Y_1}\ar[r]^{g_1}&\cdot\cdot\cdot\ar[r]^{g_{n-2}} &{Y_{n-1}}  \ar[r]^{g_{n-1}} &{Y _n}\ar[r]^{g_n}  &{Y} \ar@{-->}[r]^-{{\pi_i}_\ast\eta} &.}
$$
By the Lemma \ref {y1}, we know that $g_n$ is non-split epimorphism. Note that $f$ is right almost split, then there is a morphism $\psi_n:Y_n\rightarrow X_n$, such that $f\psi_n=g_n$. Then we have the following commutative diagram of distinguished $n$-exangles by the dual of Lemma \ref {a2}

$$\xymatrix{X_0\ar[r]^{f_0}\ar@{}[dr] \ar[d]^{\pi_i} &X_1 \ar[r]^{f_1} \ar@{}[dr]\ar[d]^{\varphi_1}&\cdot\cdot\cdot \ar[r]^{f_{n-2}} \ar@{}[dr]&X_{n-1} \ar[r]^{f_{n-1}}\ar@{}[dr]\ar[d]^{\varphi_{n-1}} &X_n \ar[r]^{f} \ar@{}[dr]\ar[d]^{\varphi_n}&Y \ar@{}[dr]\ar@{=}[d]^{} \ar@{-->}[r]^-{\eta} &\\
 K_i\ar[r]^{g_0}\ar@{}[dr] \ar@{-->}[d]^{\psi_0} &Y_1 \ar[r]^{g_1} \ar@{}[dr]\ar@{-->}[d]^{\psi_1}&\cdot\cdot\cdot \ar[r]^{g_{n-2}} \ar@{}[dr]&Y_{n-1} \ar[r]^{g_{n-1}}\ar@{}[dr]\ar@{-->}[d]^{\psi_{n-1}} &Y_n \ar[r]^{g_n} \ar@{}[dr]\ar[d]^{\psi_n}&Y \ar@{}[dr]\ar@{=}[d]^{} \ar@{-->}[r]^-{{\pi_i}_\ast\eta} &\\
{X_0}\ar[r]^{f_0} &{X_1}\ar[r]^{f_1}&\cdot\cdot\cdot\ar[r]^{f_{n-2}} &{X_{n-1}}  \ar[r]^{f_{n-1}} &{X _n}\ar[r]^{f}  &{Y} \ar@{-->}[r]^-{\eta} &.}
$$
Since $f$ is right minimal, $\psi_n\circ \varphi_n$ is an isomorphism. In a similar way of the proof in \cite[Lemma 3.12]{F}, we know that $\psi_1\circ \varphi_1,\cdot\cdot\cdot,\psi_{n-1}\circ \varphi_{n-1}$ are all isomorphism. We claim that $\psi_0\circ \pi_i$ is also an isomorphism. In fact, we have the following commutative diagram with exact rows by Lemma \ref {a1}
$$\xymatrix@C=1.2cm{
\C(X_2,-)\ar[r]^{\mathcal{C}(f_1, -)}\ar[d]^{{\C}(\psi_2\varphi_2,-)}_\cong &\C(X_1,-)\ar[r]^{\C(f_0, -)}\ar[d]^{{\C}(\psi_1\varphi_1,-)}_\cong & \C(X_0,-) \ar[r]^{~\del\ush~} \ar[d]^{\C(\psi_0\pi_i,-)}& \mathbb{E}(Y,-)\ar[r]^{\mathbb{E}(f,  -)}\ar@{=}[d]& \mathbb{E}(X_n,-)\ar[d]^{\mathbb{E}(\psi_n\varphi_n,-)}_\cong\\
\C(X_2,-)\ar[r]^{\C(f_1, -)} &\C(X_1,-)\ar[r]^{\C(f_0, -)}\ar[r] &  \C(X_0,-) \ar[r]^{~\del\ush~}& \mathbb{E}(Y,-)\ar[r]^{\mathbb{E}(f,-)}& \mathbb{E}(X_n,-)
.}$$
By the Five lemma, we have that $\C(\psi_0\pi_i,-)$ is an isomorphism, then $\psi_0\pi_i$ is an isomorphism by the Yoneda's lemma. Hence $X_0$ is a direct summand of $K_i$, which is a contradiction with our assumption. Hence $X_0$ is indecomposable.

This shows  that $X_{\bullet}:~~X_0\xrightarrow{f_0}X_1\xrightarrow{f_1}X_2\xrightarrow{f_2}\cdots\xrightarrow{f_{n-2}}X_{n-1}
\xrightarrow{f_{n-1}}X_n\xrightarrow{f}Y\overset{\eta}{\dashrightarrow}$ is an {Auslander-Reiten $n$-exangle} by Lemma \ref{r1}.
\end{proof}

{\bf Now we are ready to prove Theorem \ref{theorem2}.}
\begin{proof}(1)$\Rightarrow$(2) Suppose that the pair $(\tau_n,\varphi)$ is a Serre duality of $\C$. Let $Y$ be an object in $\C$, for any deflation $\alpha\in{\C}(X,Y)$, there exists a distinguished $n$-exangle of the form
$$X_0\xrightarrow{}X_1\xrightarrow{}X_2\xrightarrow{}\cdots\xrightarrow{}X_{n-1}
\xrightarrow{}X\xrightarrow{\alpha}Y\overset{}{\dashrightarrow}.$$
By Lemma \ref {a1}, we have an exact sequence $$\C(-, X)\xrightarrow{{\C}(-,\alpha)}
\C(-, Y)\xrightarrow{}\E(-, X_{0}).$$ By Serre duality, $\E(-,X_{0})\cong D{\C}(\tau_{n} ^{-} X_{0},-)$. It follows that there is a monomorphism

$${\rm Coker}~{{\C}(-,\alpha)}\rightarrow D{\C}(\tau_{n}^{-} X_{0},-).$$
Hence $\alpha$ is right $\tau_{n}^{-} X_{0}$-determined and $\rm(RDC1)$ holds by Proposition 5.2 in \cite{K} or Lemma 2.3 in \cite{Ch}.

For $\rm(RDC2)$, let $C$ be an object and $H$ an ${\rm{End}}_{\C}(C)^\mathrm{op}$-submodule of ${\C}(C,Y)$. Consider the morphism $\overline {\varrho}: {\C}(-,Y)\rightarrow D{\C}(C^{'},-)$ with $C^{'}\in \add C$ and $\Im\overline {\varrho}=F^{(C,H)}$ defined just before Lemma \ref{le0}. Combining $\overline {\varrho}$ with the isomorphism $\E(-,\tau_n C^{'})\cong D{\C}(C^{'},-)$  we have a morphism ${\overline {\varrho}}^{'}: {\C}(-,Y)\rightarrow \E(-,\tau_n C^{'})$ with $\Im{\overline {\varrho}}^{'}\cong \Im\overline {\varrho}=F^{(C,H)}$. For ${\overline {\varrho}}_{Y}^{'}=\delta\in\E(Y,\tau_n C^{'})$, we have a distinguished $n$-exangle of the form
$$\tau_n C^{'}\xrightarrow{}X_1\xrightarrow{}X_2\xrightarrow{}\cdots\xrightarrow{}X_{n-1}
\xrightarrow{}X\xrightarrow{\alpha}Y\overset{\delta}{\dashrightarrow}.$$
By Lemma \ref {a1}, we have an exact sequence $$\C(-, X)\xrightarrow{{\C}(-,\alpha)}
\C(-, Y)\xrightarrow{~\del\ssh~}\E(-, \tau_n C^{'}). $$
It is obvious that $(\del\ssh)_{Y}(\Id_{Y})=\delta=\overline {\varrho}_{Y}^{'}(\Id_{Y})$. Thus we have $\del\ssh=\overline {\varrho}^{'}$ by the Yoneda lemma and $\Im\del\ssh=F^{(C,H)}$. Which shows that $F^{(C,H)}$ is finitely presented. So $ \rm(RDC2)$ holds by Lemma \ref{le0}. This shows that right deflation-classified. Moreover, $\C$ has right determined deflations. Dually, one can prove other statements.

(2)$\Rightarrow$(1) Assume that $\C$ has right determined deflations and left determined inflations. Then we have $ \C=\underline{\C}=\overline{\C}$ by Remark \ref{re1}. For any indecomposable non-projective object $Y$, there exists an {Auslander-Reiten $n$-exangle} ending at $Y$ by Proposition \ref{pro1}. Dually, for any indecomposable non-injective object $X$, there exists an  {Auslander-Reiten $n$-exangle} starting at $X$. It follows that $\C$ has {Auslander-Reiten $n$-exangle}. By Theorem \ref{theorem1}, $\C$ has Auslander-Reiten-Serre duality. In particular, $\C$ has Serre duality since $ \C=\underline{\C}=\overline{\C}$.
\end{proof}

\begin{remark}In Theorem \ref{theorem2}, when $\C$ is an abelian category, it is just Theorem 3.4 in
\cite{Ch}, when $\C$ is a triangulated category, it is just Theorem 4.2 in
\cite{K}, when $\C$ is an extriangulated category, it is just Theorem 3.5 in
\cite{ZTH}.
\end{remark}

\section{The restricted Auslander bijection induced Auslander-Reiten-Serre duality }
\subsection{The Auslander bijection}
 In the section, assume further that $\C$ has Auslander-Reiten-Serre duality.

$\bullet$ We recall from \cite{R} that two morphisms $f\colon X\to Y$ and $f'\colon X'\to Y$ are called \emph{right equivalent}
if $f$ factors through $f'$ and $f'$ factors through $f$.

One can  have the following some easy observations.

\begin{remark}\label{re11} (a) A right equivalence relation is an  equivalence relation on the set of all morphisms ending in some object $Y\in\C$. We denote by $[f\rangle$ the right equivalence class of a morphism $f\in\C(X,Y)$.

(b) Assume that $f$ and $f'$ are right equivalent. Then $f$ is right $C$-determined if and only if so is $f'$. We say that $[f\rangle$ is right $C$-determined if a representative element $f$ right $C$-determined.

(c) Assume that $f$ and $f'$ are right equivalent. Then $\Im ~\C(C,f)$= $\Im ~\C(C,f')$.

(d) If $f$ and $f'$ are right $C$-determined, then $f$ and $f'$ are right equivalent if and only if $\Im ~\C(C,f)$= $\Im ~\C(C,f')$.

\end{remark}
\begin{definition}(\cite{R})
 Assume $f_{1}\in\C(X_{1},Y)$ and $f_{2}\in\C(X_{2},Y)$. Define $[f_{1}\rangle\leq [f_{2}\rangle$ provided that $f_{1}$ factors through $f_{2}$.
\end{definition}
$\bullet$ We denote by $[\rightarrow Y\rangle$ the set of right equivalence classes of morphisms to $Y$. Then $\leq$ induces a poset relation on $[\rightarrow Y\rangle$. We denote by ${^C[}\rightarrow Y\rangle$ the subset of $[\rightarrow Y\rangle$ consisting of all right equivalence class that are right $C$-determined. We denote by $\Sub_{{\End}_{\C}(C)^\mathrm{op}}{\C}(C,Y)$ the poset formed by ${\End}_{\C}(C)^\mathrm{op}$-submodules of ${\C}(C,Y)$, ordered by the inclusion. Then the following map is well-defined
$$\eta_{C,Y}:[\rightarrow Y\rangle\rightarrow \Sub_{{\End}_{\C}(C)^\mathrm{op}}{\C}(C,Y),~~~[f\rangle\mapsto\Im ~\C(C,f).$$

$(\maltese)$ The restriction of $\eta_{C,Y}$ on ${^C[}\rightarrow Y\rangle$ is injective and reflects the orders, that
is, for two classes $[f_{1}\rangle,[f_{2}\rangle\in{^C[}\rightarrow Y\rangle$, $[f_{1}\rangle\leq [f_{2}\rangle$ if and only if $\eta_{C,Y}([f_{1}\rangle)\subseteq \eta_{C,Y}([f_{2}\rangle)$.

\begin{definition}(\cite{C,R})
 If the map $\eta_{C,Y}:{^C[}\rightarrow Y\rangle\rightarrow \Sub_{{\End}_{\C}(C)^\mathrm{op}}{\C}(C,Y)$ above is surjective, then we say that the Auslander bijection at $Y$ relative to $C$ holds, or equivalently, it is an isomorphism of posets.
\end{definition}

\subsection{The restricted Auslander bijection}

Since each ${\End}_{\C}(C)^\mathrm{op}$-submodule of $\underline\C(C, Y)$ corresponds to a unique ${\End}_{\C}(C)^\mathrm{op}$-submodule of $\C(C, Y)$ containing $\P(C, Y)$, the poset $\Sub_{{\End}_{\C}(C)^\mathrm{op}}\underline{\C}(C,Y)$ is viewed as a subset of $\Sub_{{\End}_{\C}(C)^\mathrm{op}}{\C}(C,Y)$.

In what follows, we always assume that the following condition, analogous to the (WIC) Condition
in \cite[Condition 5.8]{NP}.
\begin{condition}\label{cd}
Let $f \in \C(A,B)$, $g \in\C(B,C)$ be any composable pair of morphisms.  Consider the following
conditions.

(1) If $g \circ f$ is a deflation, then so is $g$.

(2) If $g \circ f$ is an inflation, then so is $f$.
\end{condition}
Under the Condition \ref{cd}, the following result is straightforward.
\begin{lemma}\label{lem47} Suppose that $f$ and $f'$ are right equivalent. Then $f$ is a deflation if and only if $f'$ is a deflation.
\end{lemma}

Define  $$[\rightarrow Y\rangle_{\rm def}:= \{[f\rangle\in[\rightarrow Y\rangle\mid f \text{ is a deflation}\}.$$

Observe that $\P(C,Y)\subseteq\Im ~\C(C,f)$ for any $[f\rangle\in[\rightarrow Y\rangle_{\rm def}$. Then we have the following
map
$$\eta_{C,Y}:[\rightarrow Y\rangle_{\rm def}\rightarrow \Sub_{{\End}_{\C}(C)^\mathrm{op}}\underline{\C}(C,Y),~~~[f\rangle\mapsto\Im ~\C(C,f)/\P(C,Y).$$

Set ${^C[}\rightarrow Y\rangle_{\rm def}:=[\rightarrow Y\rangle_{\rm def}\cap{^C[}\rightarrow Y\rangle$. Then we have the following
map
$$\eta_{C,Y}:{^C[}\rightarrow Y\rangle_{\rm def}\rightarrow \Sub_{{\End}_{\C}(C)^\mathrm{op}}\underline{\C}(C,Y),~~~[f\rangle\mapsto\Im ~\C(C,f)/\P(C,Y),$$
which is injective by $(\maltese)$.
\begin{definition}
 If the map $\eta_{C,Y}:{^C[}\rightarrow Y\rangle_{\rm def}\rightarrow \Sub_{{\End}_{\C}(C)^\mathrm{op}}\underline{\C}(C,Y)$ above is surjective, then we say that the restricted Auslander bijection at $Y$ relative to $C$ holds, or equivalently, it is an isomorphism of posets.
\end{definition}

\subsection{A map form ${^{\tau_{n}^{-}X}[}\rightarrow Y\rangle_{\rm def}$ to  $\sub_{{\End}_{\C}(X)}{\E}(Y,X)$}

The proof of the following lemma is straightforward by {\rm (EA2$\op$)} and Lemma \ref{a2}, we omit it.
\begin{lemma}\rm\label{lem55} Let $K,Y$ be two objects in $\C$. For two given distinguished $n$-exangles $$K_0\xrightarrow{}K_1\xrightarrow{}K_2\xrightarrow{}\cdots\xrightarrow{}K_{n-1}
\xrightarrow{}K_n\xrightarrow{\alpha_1}Y\overset{\delta_1}{\dashrightarrow}$$  and $$K'_0\xrightarrow{}K'_1\xrightarrow{}K'_2\xrightarrow{}\cdots\xrightarrow{}K'_{n-1}
\xrightarrow{}K'_n\xrightarrow{\alpha_2}Y\overset{\delta_2}{\dashrightarrow},$$ consider the following statements.

\begin{enumerate}
\item[$(1)$] There is a morphism $v:K_n\rightarrow K'_n$ such that $\alpha_1=\alpha_2 v$.
\item[$(2)$] There is a morphism $u:K_0\rightarrow K'_0$ such that $\delta_2=u\sas \delta_1$.
\item[$(3)$] $\Im ~{\del\ush_2}_K\subseteq \Im ~{\del\ush_1}_K$.

\end{enumerate}
Then we have $(1)\Longleftrightarrow(2)\Longrightarrow(3)$. Moreover, if $\alpha_1$ and $\alpha_2$ are right equivalent, then $\Im ~{\del\ush_2}_K=\Im ~{\del\ush_1}_K$.

\end{lemma}

$\bullet$ Let $X,Y\in C$, and $Z_0\xrightarrow{}Z_1\xrightarrow{}Z_2\xrightarrow{}\cdots\xrightarrow{}Z_{n-1}
\xrightarrow{}W\xrightarrow{f}Y\overset{\delta_f}{\dashrightarrow}$ be a distinguished $n$-exangle. By Definition $\ref {def93}$, it is easy to see that $\Im ~{\del\ush_f}_X$ is an ${\End}_{\C}(X)$-submodule of $\E(Y,X)$. By Lemma $\ref {lem55}$, we know that the following map is well-defined
$$\xi_{X,Y}:{[}\rightarrow Y\rangle_{\rm def}\rightarrow \Sub_{{\End}_{\C}(X)}{\E}(Y,X),~~~[f\rangle\mapsto\Im ~{\del\ush_f}_X.$$

$\bullet$ We denote by ${{_X}[}\rightarrow Y\rangle_{\rm def}$ the subset of ${[}\rightarrow Y\rangle_{\rm def}$ consisting of those classes $[f\rangle$ that have a representative element $f$ such that there exists a distinguished $n$-exangle
$$X_0\xrightarrow{}X_1\xrightarrow{}X_2\xrightarrow{}\cdots\xrightarrow{}X_{n-1}
\xrightarrow{}W\xrightarrow{f}Y\overset{\delta_f}{\dashrightarrow}$$ with $X_0\in\add X$. In this case, $\C (X_0,X)$ is a finitely generated projective ${\End}_{\C}(X)$-module, and hence $\xi_{X,Y}([f\rangle)=\Im ~{\del\ush_f}_X$ is a finitely generated ${\End}_{\C}(X)$-module.

$\bullet$ We denote by $\sub_{{\End}_{\C}(X)}{\E}(Y,X)$ the subset of $\Sub_{{\End}_{\C}(X)}{\E}(Y,X)$ consisting of finitely generated ${\End}_{\C}(X)$-modules. Then the $\xi_{X,Y}$ induces a  well-defined map which we still denote by  $\xi_{X,Y}$
$$\xi_{X,Y}:{{_X}[}\rightarrow Y\rangle_{\rm def}\rightarrow \sub_{{\End}_{\C}(X)}{\E}(Y,X),~~~[f\rangle\mapsto\Im ~{\del\ush_f}_X.$$
Moreover, we have the following lemma.

\begin{lemma}\rm\label{lem9} The map $$\xi_{X,Y}:{{_X}[}\rightarrow Y\rangle_{\rm def}\rightarrow \sub_{{\End}_{\C}(X)}{\E}(Y,X),~~~[f\rangle\mapsto\Im ~{\del\ush_f}_X$$ is an anti-isomorphism of posets.
\end{lemma}
\proof Since the proof is very similar to \cite[Theorem 4.1]{ZTH}, we omit it. \qed
\begin{lemma}\label{theorem4}
Let $$X\xrightarrow{\alpha}X_1\xrightarrow{\alpha_1}X_2\xrightarrow{\alpha_2}\cdots\xrightarrow{\alpha_{n-2}}X_{n-1}
\xrightarrow{\alpha_{n-1}}Z\xrightarrow{\beta}Y\overset{\delta}{\dashrightarrow}$$ be a distinguished $n$-exangle. Then

\begin{enumerate}
\item[$(1)$] $\beta$ is right $\tau_{n}^- X$-determined.
\item[$(2)$] If $\alpha$ is in ${\rm rad}_{\C}$, then $\beta$ is right $C$-determined for some $C\in\C$ if and only if $\tau_{n}^- X\in\rm add C$.

Consequently, we have ${{_X}[}\rightarrow Y\rangle_{\rm def}={^{\tau_{n}^{-}X}[}\rightarrow Y\rangle_{\rm def}$.

\end{enumerate}
 \end{lemma}
 \begin{proof}
$(1)$ It follows from  \cite[Lemma 4.6]{HHZ}.

$(2)$ $``\Leftarrow "$ It follows from (1).

$``\Rightarrow "$ We  will show that each indecomposable direct summand $X'$ of $X$ satisfies $\tau_{n}^- X'\in\rm add C$. Firstly, we claim that the composition of inflations $X'\xrightarrow{\iota} X\xrightarrow{\alpha} X_1$ is not a split monomorphism, where $\iota$ is the natural projection. If not, assume that $\alpha\iota$ is a split monomorphism. There exists a morphism $t: X_1\xrightarrow{} X'$, such that $t\alpha\iota=1$. We have $t\alpha\iota\in\rad_{\C}$ since $\alpha$ is in $\rad_{\C}$. This shows $1-t\alpha\iota$ is invertible. Which is a contradiction since $1-t\alpha\iota=0$. Moreover,  $X'$ is not an injective object by the dual of \cite[Lemma 3.4]{LZ}. Hence we have an {Auslander-Reiten $n$-exangle} of the form by Lemma \ref{02}
$$X'\xrightarrow{\alpha'}W_1\xrightarrow{\alpha'_{1}}W_2\xrightarrow{}\cdots\xrightarrow{\alpha'_{n-1}}W_n\xrightarrow{\beta'}\tau_{n}^- X'\overset{\sigma}{\dashrightarrow}.$$
Since $\alpha'$ is left almost split and $\alpha i$ is not a split monomorphism, there exists a morphism $i_1:W_1\rightarrow X_1$, such that $i_1\alpha'=\alpha i$. We have the following commutative diagram by Lemma \ref{a2}
$$\xymatrix{X'\ar[r]^{\alpha'}\ar@{}[dr] \ar[d]^{\iota} &W_1 \ar[r]^{\alpha'_1} \ar@{}[dr]\ar[d]^{i_1}&\cdot\cdot\cdot \ar[r]^{\alpha'_{n-2}} \ar@{}[dr]&W_{n-1} \ar[r]^{\alpha'_{n-1}}\ar@{}[dr]\ar@{-->}[d]^{i_{n-1}} &W_n \ar[r]^{\beta'} \ar@{}[dr]\ar@{-->}[d]^{i_n}&\tau_{n}^- X' \ar@{}[dr]\ar@{-->}[d]^{i_{n+1}} \ar@{-->}[r]^-{\sigma} &\\
{X}\ar[r]^{\alpha} &{X_1}\ar[r]^{\alpha_1}&\cdot\cdot\cdot\ar[r]^{\alpha_{n-2}} &{X_{n-1}}  \ar[r]^{\alpha_{n-1}} &{Z}\ar[r]^{\beta}  &{Y} \ar@{-->}[r]^-{\delta} &}
$$
with $\iota_\ast\sigma=i_{n+1}^\ast\delta$.

If $\tau_{n}^- X' \not\in\rm add C$, then any $f\in\C(C,\tau_{n}^- X')$ is not a split epimorphism. Hence there is a morphism $g:C\rightarrow W_{n}$, such that $\beta'g=f$. So $i_{n+1}f=i_{n+1}(\beta'g)=\beta(i_{n}g)$. Moreover, since $\beta$ is right $C$-determined, there exists a morphism $h:\tau_{n}^- X'\rightarrow Z$ such that $i_{n+1}=\beta h$.

Consider the following commutative diagram by \rm (EA2)
$$\xymatrix{X\ar[r]^{\gamma_0}\ar@{}[dr] \ar@{=}[d] &W'_1 \ar[r]^{\gamma_1} \ar@{}[dr]\ar@{-->}[d]^{i_1}&\cdot\cdot\cdot \ar[r]^{\gamma_{n-2}} \ar@{}[dr]&W'_{n-1} \ar[r]^{\gamma_{n-1}}\ar@{}[dr]\ar@{-->}[d]^{i_{n-1}} &W'_n \ar[r]^{\gamma_n} \ar@{}[dr]\ar@{-->}[d]^{i_n}&\tau_{n}^- X' \ar@{}[dr]\ar[d]^{i_{n+1}} \ar@{-->}[r]^-{i_{n+1}^\ast\delta} &\\
{X}\ar[r]^{\alpha} &{X_1}\ar[r]^{\alpha_1}&\cdot\cdot\cdot\ar[r]^{\alpha_{n-2}} &{X_{n-1}}  \ar[r]^{\alpha_{n-1}} &{Z}\ar[r]^{\beta}  &{Y} \ar@{-->}[r]^-{\delta} &.}
$$
Since $i_{n+1}$ factors through $\beta$, we have $\id_{X}$ factors through $\gamma_0$ and hence $\gamma_0$ is split monomorphism by Lemma \ref{y1}. Moreover, $\iota_\ast\sigma=i_{n+1}^\ast\delta=0$.

Consider the following commutative diagram by {\rm (EA2$\op$)}
$$\xymatrix{X'\ar[r]^{\alpha'}\ar@{}[dr] \ar[d]^{\iota} &W_1 \ar[r]^{\alpha'_1} \ar@{}[dr]\ar@{-->}[d]^{}&\cdot\cdot\cdot \ar[r]^{\alpha'_{n-2}} \ar@{}[dr]&W_{n-1} \ar[r]^{\alpha'_{n-1}}\ar@{}[dr]\ar@{-->}[d]^{i_{n-1}} &W_n \ar[r]^{\beta'} \ar@{}[dr]\ar@{-->}[d]^{i_n}&\tau_{n}^- X' \ar@{}[dr]\ar@{=}[d]^{} \ar@{-->}[r]^-{\sigma} &\\
{X}\ar[r]^{\alpha} &{W''_1}\ar[r]^{\alpha_1}&\cdot\cdot\cdot\ar[r]^{\alpha_{n-2}} &{W''_{n-1}}  \ar[r]^{\alpha_{n-1}} &{W''_{n}}\ar[r]^{\beta}  &{\tau_{n}^- X'} \ar@{-->}[r]^-{\iota_\ast\sigma} &.}
$$
Since $\iota_\ast\sigma=0$, by Lemma \ref{y1}, we know that there exists a morphism $\omega:W_1\rightarrow X$ such that $\iota=\omega\alpha'$. Note that $\iota$ is a split monomorphism, $\alpha'$ is also a a split monomorphism, which is a contradiction. Therefore, $\tau_{n}^- X'\in\rm add C$.
\end{proof}

\begin{remark}\label{r01} Let $L \xrightarrow{\alpha} M \xrightarrow{\beta}N$ be a complex such that the sequence $$\C(M,L) \xrightarrow{} \C(M,M) \xrightarrow{}\C(M,N)$$ is exact. Then $\alpha$ is in $\rad_{\C}$ if and only if $\beta$ is right minimal (see \cite[Lemma 1.1]{JK} ).
Thus if $X\xrightarrow{\alpha}Y\xrightarrow{\beta}Z\overset{\delta}{\dashrightarrow}$
is an $\E$-triangle in an extriangulated category, then
$\alpha\in\rad_{\C}$ if and only if $\beta$ is right minimal.
 Hence Lemma \ref{theorem4} is a higher counterpart of \cite[Proposition 4.2]{ZTH}.

\end{remark}

\begin{theorem}\rm\label{lem6} The map $$\xi_{X,Y}:{^{\tau_{n}^{-}X}[}\rightarrow Y\rangle_{\rm def}\rightarrow \sub_{{\End}_{\C}(X)}{\E}(Y,X),~~~[f\rangle\mapsto\Im ~{\del\ush_\alpha}_X$$ is an anti-isomorphism of posets.
\end{theorem}
\begin{proof}
It follows from Lemma \ref{lem9} and Lemma \ref{theorem4}.
\end{proof}
\subsection{The restricted Auslander bijection induced by Auslander-Reiten-Serre duality}

\begin{lemma}\rm\label{lem14} Let $X, Y $ be objects in $\C$. There is a bijection $$\Upsilon_{X,Y}:\sub_{{\End}_{\C}(X)}{\E}(Y,X)\rightarrow \sub_{{\End}_{\C}(X)^\mathrm{op}}\underline{\C}(\tau_{n}^{-}X,Y)$$ such that for any ${\End}_{\C}(X)$-submodule $F$ of ${\E}(Y,X)$, $\Upsilon_{X,Y}(F)=H$ is defined by an exact sequence
$0\xrightarrow{} H\xrightarrow{}\underline{\C}(\tau_{n}^{-}X,Y)\xrightarrow{D(i)\Psi_{Y,X}^{-1}}DF\xrightarrow{}0 $, where $i:F\rightarrow{\E}(Y,X)$ is the inclusion. The bijection $\Upsilon_{X,Y}$ is an anti-isomorphism of posets.
\end{lemma}
\proof Since the proof is very similar to \cite[Lemma 5.1]{ZTH}, we omit it. Moreover, one also can see \cite[Lemma 4.2]{C}. \qed

$\bullet$ For any $X\in C$, there are natural isomorphisms

$$\Phi_{X,-}^{-1}: \overline{{\C}}(-,\tau_{n} X)\rightarrow D\E(X,-)$$ and $$\Psi_{-,X}^{-1}:{\underline{\C}}(\tau_{n}^{-}X, -)\rightarrow D\E(-,X).$$

Set $$\lambda_X:=\Phi_{X,\tau_{n}{X}}^{-1}(\overline{\Id_{\tau_{n}{X}}})\in D\E(X,\tau_{n}{X}),~~~ \underline{\mu_X}:=\Psi_{X,\tau_{n}{X}}(\lambda_X)\in \underline{\C}(\tau_{n}^{-}\tau_{n}X,{X}),$$
$$\kappa_X:=\Psi_{\tau_{n}^{-}{X},X}^{-1}(\underline{\Id_{\tau_{n}^{-}{X}}})\in D\E(\tau_{n}^{-}{X},X),~~~\overline{\nu_X}:=\Phi_{\tau_{n}^{-}{X},X}(\kappa_X)\in \overline{\C}({X},\tau_{n}\tau_{n}^{-}X).$$

Let $X_0\xrightarrow{}X_1\xrightarrow{}X_2\xrightarrow{}\cdots\xrightarrow{}X_{n-1}
\xrightarrow{}X_{n}\xrightarrow{}Y\overset{\delta}{\dashrightarrow}$ be a distinguished $n$-exangle. Then for any $X\in C$, we have the following two commutative diagrams (more details can see \cite{ZTH})
 \[\xymatrix@C+5em{
   D\E( X,X_0)\ar[r]^-{D(\del\ssh)_X}
      &D\underline{\C}( X,Y)\\
    \overline{\C}(X_{0},\tau_n X)\ar[r]^-{{\del\ush}_{\tau_nX}}\ar[u]_-{\Phi_{X,X_0}^{-1}}
      &\E({ Y},\tau_nX)\ar[u]_-{D(\Psi_{Y,\tau_n X}^{-1}\underline{\C}(\underline{\mu_X},Y))}
  }\]
  and
 \[\xymatrix@C+5em{
   D\E( Y,X)\ar[r]^-{D{\del\ush}_X}
      &D\overline{\C}( X_{0},X)\\
    \underline{\C}(\tau_{n}^{-}{X},Y)\ar[r]^-{{(\del\ssh)}_{\tau_{n}^{-}{X}}}\ar[u]_-{\Psi_{Y,X}^{-1}}
      &\E(\tau_{n}^{-}{X},X_0)\ar[u]_-{D(\Phi_{\tau_{n}^{-}{X},X_0}^{-1}\overline{\C}(X_0,\overline{\nu_X}))}.
  }\]
\begin{remark}\label{remark14} By the two commutative diagrams as above, it is easy to see that there are exact sequences  $$0\xrightarrow{} \Ker(\del\ush)_{\tau_{n}{X}}\xrightarrow{}\overline{\C}(X_0,\tau_{n}X)\xrightarrow{D(i_1)\Phi_{X,X_0}^{-1}}D\Im (\del\ssh)_{X} \xrightarrow{}0 $$
and
$$0\xrightarrow{} \Ker(\del\ssh)_{\tau_{n}^{-}{X}}\xrightarrow{}\underline{\C}(\tau_{n}^{-}{X},Y)\xrightarrow{D(i_2)\Psi_{Y,X}^{-1}}D\Im \del\ush_{X} \xrightarrow{}0 ,$$
where $i_1:\Im (\del\ssh)_{X}\rightarrow\E( X,X_0)$ and $i_2:\Im \del\ush_{X}\rightarrow\E( Y,X)$ are the corresponding inclusions.
\end{remark}
$\bullet$ For any $C,Y\in \C$, we have a well-defined map $$\eta_{C,Y}:[\rightarrow Y\rangle_{\rm def}\rightarrow \Sub_{{\End}_{\C}(C)^\mathrm{op}}\underline{\C}(C,Y),~~~[f\rangle\mapsto\Im \underline{\C}(C,f) ~$$ we observe that $\Im \underline{\C}(C,f)=\Im\C(C,f)/\P(C,Y)$.

For any $X\in \C$, since $\tau_{n}^{-}$ is an equivalence, we can identity via $\tau_{n}^{-}$ the ${{\End}_{\C}(\tau_{n}^{-}X)^\mathrm{op}}$-module structure on $\underline{\C}(\tau_{n}^{-}X,Y)$ with the corresponding ${{\End}_{\C}(X)^\mathrm{op}}$-module structure. Hence, we can identity the poset $\Sub_{{\End}_{\C}(\tau_{n}^{-}X)^\mathrm{op}}\underline{\C}(\tau_{n}^{-}X,Y)$ with $\Sub_{{\End}_{\C}(X)^\mathrm{op}}\underline{\C}(\tau_{n}^{-}X,Y)$. By the identification, we have the bijection
$$\Upsilon_{X,Y}:\sub_{{\End}_{\C}(X)}{\E}(Y,X)\rightarrow \sub_{{\End}_{\C}(\tau_{n}^{-}X)^\mathrm{op}}\underline{\C}(\tau_{n}^{-}X,Y).$$

\begin{lemma}\rm\label{lem1u} Let $X, Y $ be objects in $\C$. Then the following
triangle is commutative
$$\xymatrix{
&  \sub_{{\End}_{\C}(\tau_{n}^{-}X)^\mathrm{op}}\underline{\C}(\tau_{n}^{-}X,Y) & \\
{[}\rightarrow Y\rangle_{\rm def}\ar[ur]^{\eta_{\tau_{n}^{-}X,Y}}  \ar[rr]^{\xi_{X,Y}} &  &\sub_{{\End}_{\C}(X)}{\E}(Y,X)\ar[ul]_{\Upsilon_{X,Y}}.
}
$$

\end{lemma}
\proof For any $[f\rangle\in{[}\rightarrow Y\rangle_{\rm def} $, there is a distinguished $n$-exangle $$X_0\xrightarrow{}X_1\xrightarrow{}X_2\xrightarrow{}\cdots\xrightarrow{}X_{n-1}
\xrightarrow{}X_{n}\xrightarrow{f}Y\overset{\delta}{\dashrightarrow}.$$
We have an exact sequence
$$\underline{\C}(\tau_{n}^{-}X,X_{n})\xrightarrow{\underline{\C}(\tau_{n}^{-}X,f)}\underline{\C}(\tau_{n}^{-}X,Y)\xrightarrow{(\del\ssh)_{\tau_{n}^{-}X}}\E(\tau_{n}^{-}X,X_0).$$
By definition, we have $\eta_{\tau_{n}^{-}X,Y}([f\rangle)=\Im\underline{\C}(\tau_{n}^{-}X,f)=\Ker(\del\ssh)_{\tau_{n}^{-}X}$ and $\xi_{X,Y}([f\rangle)=\Im ~{\del\ush}_X$.
It follows that $\Upsilon_{X,Y}(\Im ~{\del\ush}_X)=\Ker(\del\ssh)_{\tau_{n}^{-}X}$ by Lemma \ref{lem14} and Remark \ref{remark14}. Thus we have $\eta_{\tau_{n}^{-}X,Y}=\Upsilon_{X,Y}\xi_{X,Y}$.

   \qed

Next we are ready to state and prove our third main result.

\begin{theorem}\rm\label{thrm} Let $X, Y $ be objects in $\C$. Then the following
triangle is commutative
$$\xymatrix{
&  \sub_{{\End}_{\C}(\tau_{n}^{-}X)^\mathrm{op}}\underline{\C}(\tau_{n}^{-}X,Y) & \\
{{_X}[}\rightarrow Y\rangle_{\rm def}={^{\tau_{n}^{-}X}[}\rightarrow Y\rangle_{\rm def}\ar[ur]^{\eta_{\tau_{n}^{-}X,Y}}  \ar[rr]^{\xi_{X,Y}} &  &\sub_{{\End}_{\C}(X)}{\E}(Y,X)\ar[ul]_{\Upsilon_{X,Y}}.
}
$$
In particular, we have the restricted Auslander bijection at $Y$ relative to $\tau_{n}^{-}X$
$$\eta_{\tau_{n}^{-}X,Y}: {^{\tau_{n}^{-}X}[}\rightarrow Y\rangle_{\rm def}\rightarrow  \sub_{{\End}_{\C}(\tau_{n}^{-}X)^\mathrm{op}}\underline{\C}(\tau_{n}^{-}X,Y) ,              $$
which is an isomorphism of posets.
\end{theorem}
\proof It follows from Theorem \ref{lem6}, Lemma \ref{lem14} and Lemma \ref{lem1u}. \qed

\begin{remark}In Theorem \ref{thrm}, when $\C$ is an abelian category, it is just Theorem 4.6 in
\cite{C}, when $\C$ is an extriangulated category, it is just Theorem 5.4 in
\cite{ZTH}.
\end{remark}

\textbf{Jian He}\\
Department of Mathematics, Nanjing University, 210093 Nanjing, Jiangsu, P. R. China\\
E-mail: \textsf{jianhe30@163.com}\\[0.3cm]
\textbf{Jing He}\\
College of Science, Hunan University of Technology and Business, 410205 Changsha, Hunan P. R. China\\
E-mail: \textsf{jinghe1003@163.com}\\[0.3cm]
\textbf{Panyue Zhou}\\
College of Mathematics, Hunan Institute of Science and Technology, 414006 Yueyang, Hunan, P. R. China.\\
E-mail: \textsf{panyuezhou@163.com}

\end{document}